\documentclass{amsart}
\usepackage{amsrefs,graphicx}

\usepackage{amsmath,amsfonts,graphics,color, hyperref}
\usepackage{amssymb}

\numberwithin{equation}{section}
\newtheorem{theorem}{Theorem}

\newtheorem{prop}{Proposition}
\newtheorem{lemma}{Lemma}
\newtheorem{cor}{Corollary}

\newtheorem{remark}{Remark}

\newcommand{\vphi}{\varphi}

\newcommand{\pa}{\partial}
\newcommand{\eps}{\varepsilon}

\newcommand{\n}{\nabla}
\newcommand{\e}{\epsilon} 

\newcommand{\N}{\mathbb{N}}

\newcommand{\R}{\mathbb{R}}

\newcommand{\cC}{\mathcal{C}}
\newcommand{\cL}{\mathcal{L}}
\renewcommand{\epsilon}{\varepsilon}

        \definecolor{pink}{rgb}{1,0,1}

\begin{document} 

\title{Eigenvalues of collapsing domains and drift Laplacians}
\author[Zhiqin Lu]{Zhiqin Lu}
\address{Department of Mathematics\\ University of California, Irvine, USA.}  \email{zlu@math.uci.edu}  
\author[Julie Rowlett]{Julie Rowlett}
\address{Max Planck Institut f\"ur Mathematik\\ Vivatgasse 7, Bonn Germany.} \email{rowlett@mpim-bonn.mpg.de}
\date{}
 

\begin{abstract}
By introducing a weight function to the Laplace operator, Bakry and \'Emery defined the ``drift Laplacian'' to study diffusion processes.  Our first main result is that, given a Bakry-\'Emery manifold, there is a naturally associated family of graphs whose eigenvalues converge to the eigenvalues of the drift Laplacian as the graphs collapse to the manifold.  Applications of this result include a new relationship between Dirichlet eigenvalues of domains in $\R^n$ and Neumann eigenvalues of domains in $\R^{n+1}$ and a new maximum principle.  Using our main result and maximum principle, we are able to generalize \emph{all the results in Riemannian geometry based on gradient estimates to Bakry-\'Emery manifolds}.
\end{abstract}

\maketitle

\section{Introduction}
Bakry-\'Emery geometry was introduced in \cite{BE} to study
diffusion processes.  For a Riemannian manifold $(M, g)$ and $\phi \in \cC^2 (M)$, the Bakry-\'Emery manifold is a triple
$(M,g,\phi)$, where the measure on $M$ is the weighted measure
$e^{-\phi} dV_g$.  If ${\rm Ric}$ and $\Delta$ are, respectively, the Ricci curvature and Laplacian with respect to the Riemannian metric $g$, then the Bakry-\'Emery Ricci curvature is defined to
be\footnote{In the notation of \cite{lott}, this is the $\infty$
Bakry-\'Emery Ricci curvature.}
\[
{\rm Ric}_\infty={\rm Ric}+{\rm Hess}(\phi),
\]
and the Bakry-\'Emery Laplacian is
\[
\Delta_\phi=\Delta - \nabla\phi\cdot\nabla.
\]  
The operator can be extended as  a self-adjoint operator with respect
to the weighted measure $e^{-\phi} dV_g$; it is also known as a ``drifting'' or ``drift'' Laplacian.

\begin{theorem}\label{thm1}   
Let $(M, g, \phi)$ be a compact Bakry-\'Emery manifold.  Let 
$$M_{\eps} := \{ (x, y)\mid x \in M, \quad 0 \leq y \leq \eps e^{-\phi(x)} \} \subset M \times \R^+,$$
with $\phi \in \cC^2 (M)$ and $e^{-\phi} \in \cC (M \cup \pa M)$.  Let $\{\mu_k \}_{k=0} ^{\infty}$ be the eigenvalues of the Bakry-\'Emery Laplacian on $M$.  If $\pa M \neq \emptyset$, assume the Neumann boundary condition.  Let $\mu_{k}(\eps)$ be the Neumann eigenvalues of $M_\eps$ for $\tilde{\Delta} := \Delta + \pa_y ^2$.  Then $ \mu_k(\eps)=\mu_k+O(\eps^2)$ for $k\geq 0$.
\end{theorem}  

A corollary of Theorem~\ref{thm1} gives a relationship between the Dirichlet eigenvalues in $\R^n$ and Neumann eigenvalues in $\R^{n+1}$.  

\begin{cor}\label{cor-11}
Let $M$ be a bounded domain in $\mathbb R^n$ with smooth boundary, and let $\phi_1$ be
the first Dirichlet eigenfunction of the Euclidean Laplacian on
$M$.  Define
$$M_\eps :=\{(x,y)\in\mathbb R^{n+1}\mid x\in M,0\leq y\leq
\eps\phi_1(x)^2\}.$$
Let $\{ \lambda_k\}_{k=1} ^{\infty}$ be the Dirichlet eigenvalues of
$M$, and let $\{ \mu_{k}(\eps) \}_{k=0} ^{\infty}$ be the Neumann
eigenvalues of $M_\eps$. Then $\lim_{\eps \to 0} \mu_{k-1}(\eps) = \lambda_k-\lambda_1$, for all $k \in \N$.  
\end{cor}

In the second part of the paper, we establish a new maximum principle which, together with
 Theorem~\ref{thm1}, imply the following.    

\medskip

{\bf Principle.} {\em There is a one-one correspondence between the gradient estimate on a Riemannian manifold and on a Barky-\'Emery manifold. More precisely, the  eigenvalue estimate on the Bakry-\'Emery manidold $(M,g,\phi)$ is equivalent to that on the Riemannian manifold $(M_\eps, g+dy^2)$ for $\eps$ small enough.  } 

\medskip

The method of gradient estimates in eigenvalue problems was first used by Li-Yau~\cite{li-yau}. The papers~\cites{swyy,yau,yau-3,kroger,bakry-qian,zhong-yang,ac,yau-4} are the most influential to this work.  Gradient estimates on Riemannian manifolds are often quite complicated.
The point of the above {\bf Principle} is that one may apply all the proofs of gradient estimates directly to Bakry-\'Emery geometry without repeating the calculations.
\\

We end this section by discussing an alternative but related setting.\footnote{We are grateful to O. Munteanu for bringing this reference to our attention.}  In ~\cite{lott}, Lott considered the product 
$M\times S^q$ for $q>0$, with warped product metric 
\[
g_i^{M\times S^q}=g^M+\frac{1}{i^2}\phi^{\frac 2q}g^{S^q}.
\]
A curvature bound of the type 
\[
{\rm Ric}_\phi^q={\rm Ric}^M+{\rm Hess}(\phi)-\frac 1q d\phi\otimes d\phi\geq r g^M
\]
for some constant $r$ implies a lower bound of the curvature on $M\times S^q$:
\[
{\rm Ric}^{M\times S^q}\geq r_1 g^{M\times S^q}
\]
for a possibly smaller $r_1$, as $i\to\infty$.  In this setting, the manifolds $(M_i, g_i)$ converge with respect to Gromov-Hausdorff norm to the manifold $(M, g^M)$.  On the one hand, $q \in \N$ allows higher dimensional collapse, whereas we consider only one collapsing direction.  On the other hand, the fibers in ~\cite{lott} are {\it closed,} which eliminates the need for estimates regarding the boundary of $M_\eps$.  In order to prove convergence of the eigenvalues in the setting of ~\cite{lott}, one could apply Theorem 7.3 from Cheeger and Colding ~\cite{cc}, but one must first prove that the gradient of the $k^{th}$ eigenfunction is bounded above {\it uniformly in $i$} (c.f. equation (7.4) of ~\cite{cc}).  Moreover, to prove the {\it uniform} eigenvalue convergence (for each $k$) as done here, one would need to demonstrate technical results analogous to those in \S \ref{sec5}.  {In the case of one collapsing direction, the $L^2$ bounds on the eigenfunctions are sufficient to prove uniform convergence of the eigenvalues.}  When more than one dimension collapses, some of the arguments in \S \ref{sec3}  will not apply.  However, in the closed case, further techniques are available and boundary estimates are obsolete, so it may be possible to generalize our arguments.  This will be carried out elsewhere.  

The eigenvalues of the Laplacian on manifolds {\it with boundary } which collapse or are ``very thin in one direction'' has been studied by several authors.  An estimate and convergence result for the Dirichlet Laplacian on planar domains {were} obtained by ~\cite{gj} and later refined in ~\cite{f-s}.  In all dimensions, there is motivation from physics to understand the spectrum of the Laplacian for ``thin tubes,'' see \cite{dg}.  Our work can be seen as the higher dimensional Neumann analogue of ~\cite{gj} in the more general setting of Bakry-\'Emery geometry.  It would be interesting to determine further terms in the expansion of $\mu_k(\eps)$ as $\eps \to 0$ in the spirit of Friedlander and Solomyak's work in two dimensions ~\cite{f-s}.  This investigation will be the subject of forthcoming work.

The organization of the paper is as follows.  In \S 2, we present the variational principles for the drift Laplacian which we use heavily in our proof.  The proof of Theorem~\ref{thm1} and Corollary~\ref{cor-11} comprise \S 3.  We prove the new maximum principle and discuss its applications in \S 4; finally, \S5 contains technical results on Schauder estimates  which are of independent interest.  

\section{Variational principles}
On a Riemannian manifold $(M, g)$ with boundary $\pa M$, the Laplace operator can be written as 
$$\Delta = \frac{1}{\sqrt{\det(g)}} \sum_{i,j} \pa_i g^{ij} \sqrt{\det(g)} \pa_j,$$ and in particular on $\R^n$ with the Euclidean metric, 
$$\Delta =  \sum_{j=1} ^n \frac{\pa^2 }{\pa x_j ^2}.$$
The Dirichlet (respectively, Neumann) eigenvalues of the Laplace operator are the real numbers $\lambda$ for which there exists an eigenfunction
$$u \in \cC^{\infty} (M) \textrm{ such that } -\Delta u = \lambda u \textrm{ and } \left. u \right|_{\pa M} = 0, \textrm{ (respectively, } \left. \frac{\pa u}{\pa n} \right|_{\pa M} =0).$$
{The eigenvalues of the drift Laplace operator are defined analogously.}  

We shall use $\lambda$ to denote Dirichlet eigenvalues, $\mu$ to denote Neumann eigenvalues, and index the Dirichlet eigenvalues by $\N$ and the Neumann eigenvalues by $0 \cup \N$.  The Dirichlet and Neumann\footnote{Note that the Neumann boundary condition is automatically satisfied if no boundary condition is imposed in the variational principle.}  eigenvalues, respectively, satisfy the following variational principles \cite{chavel},
$$\lambda_k = \inf_{\vphi \in \cC^1 (M)} \left\{ \left. \frac{ \int_{M} |\nabla \vphi |^2 }{\int_{M} \vphi^2 } \, \right| \, \left. \,  \vphi \right|_{\pa M} = 0, \, \vphi \not\equiv 0 = \int_{M} \vphi \phi_l, \, 0 \leq  l< k \right\},$$
$$\mu_j = \inf_{\vphi \in \cC^1 (M)} \left\{   \left. \frac{ \int_{M} |\nabla \vphi |^2 }{\int_{M} \vphi^2 } \, \right| \, \,  \vphi \not\equiv 0 = \int_{M} \vphi \varphi_l, \, -1 \leq l < j \right\},$$
for $k\geq 1$ and $j\geq 0$,
where $\phi_j$ and $\varphi_l$ are, respectively, eigenfunctions for $\lambda_j$ and $\mu_l$ (assuming that
$\phi_0\equiv 0$ and $\varphi_{-1}\equiv 0$).

\begin{prop}\label{prop3-1}  Let $(M, g, \phi)$ be a Bakry-\'Emery manifold with boundary.  Then, the Dirichlet and Neumann eigenvalues of the associated drift Laplacian satisfy the following variational principles.  
$$\lambda_k =  \inf_{\varphi \in \cC^1(M)}  \left \{ \left . \frac{\int_M|\nabla\varphi|^2 e^{-\phi} }{\int_M\varphi^2 e^{-\phi}} \, \right| \varphi |_{\pa M} = 0, \, \,  \varphi \not\equiv 0 = \int_{M} \varphi \phi_j e^{-\phi}, \, \, 0 \leq j < k, \, \right\},
$$
$$\mu_j =  \inf_{\varphi \in \cC^1(M)}  \left \{ \left . \frac{\int_M|\nabla\varphi|^2 e^{-\phi} }{\int_M\varphi^2 e^{-\phi}} \, \right|   \, \,  \varphi \not\equiv 0 = \int_{M} \varphi \varphi_l e^{-\phi}, \, -1 \leq l< j\right\}
$$
for $k\geq 1$ and $j\geq 0$.
\end{prop}

\begin{remark} Let $\{ \lambda_k \}_{k=1} ^{\infty}$  be the Dirichlet eigenvalues and let the associated orthonormal basis of eigenfunctions be $\{ \phi_k \}_{k=1} ^{\infty}$. Setting the weight function $\phi = - 2 \log \phi_1$, the variational principle for $(M, g, \phi)$ is that  for all $k \geq 1$, 
$$\lambda_k - \lambda_1 =  \inf_{\varphi \in \cC^1(M)}  \left \{ \left . \frac{\int_M|\nabla\varphi|^2\phi_1^2}{\int_M\varphi^2\phi_1^2} \, \right|  \, \,  \varphi \not\equiv 0 = \int_{M} \varphi \phi_j \phi_1, \, 0 \leq j < k \right\}.$$
\end{remark} 

When $k=2$, and the domain $M \subset \R^n$, the following variational principle is Corollary 1.3 of \cite{kir-sim} and is based on results of \cite{ds}.  
The following proposition is a useful tool.  

\begin{prop}\label{prop4} For $k \geq 1$, let $\xi_0,\cdots,\xi_{k-1}$ be a nontrivial orthogonal set with respect to the weighted $\cL^2$ measure; that is 
\[
\int_M \xi_i\xi_j e^{-\phi}=0
\]
for $i\neq j$ and $\xi_i\not\equiv 0$. Then{,} we have for the Neumann eigenvalues 
\[
\sum_{j=0}^k \mu_j\leq\sum_{j=0}^{k}\frac{\int_M |\n\xi_j|^2 e^{-\phi}}{\int_M |\xi_j|^2\phi_1^2}.
\]
\end{prop} 

The proof is well know{n} and is omitted. \qed

We demonstrate that the difference between the $k^{th}$ and {the first} Dirichlet eigenvalues is the Neumann eigenvalue of a certain drift Laplacian.  This result was known to {Singer-Wong-Yau-Yau} \cite{swyy}.

\begin{prop} \label{prop2}
For a bounded domain $M \subset \R^n$, let $\{ \lambda_k\}_{k=1}
^{\infty}$ be the Dirichlet eigenvalues of the Euclidean Laplacian with orthonormal basis of eigenfunctions $\{ \phi_k \}_{k=1} ^{\infty}$,  and let $\{\mu_k\}_{k=0} ^{\infty}$ be the Neumann eigenvalues of the drift Laplacian on $M$ with respect to the weight function $- 2 \log \phi_1$.  Then, $\lambda_k - \lambda_1 = \mu_{k-1}$ for all $k \in \N$.  
\end{prop}

\begin{proof} This follows from the following formula (cf. ~\cite{swyy})
\[
\Delta\left(\frac{\phi_k}{\phi_1}\right)+2\n\log\phi_1\n\left(\frac{\phi_k}{\phi_1}\right)=-(\lambda_k-\lambda_1)\left(\frac{\phi_k}{\phi_1}\right).
\]
\end{proof}  

Finally, throughout this paper we will use the following notations:  for a function $f(t)$ and fixed $k \geq 0$, 
$$f(t) = O(t^k) \textrm{ as } t \to 0 \textrm{ if there exists } C, \delta > 0 \textrm{ such that } |f(t)| \leq Ct^k \textrm{ for all } |t| \leq \delta;$$
$$f(t) = o(t^k) \textrm{ as } t \to 0 \textrm{ if } \lim_{t \to 0} \frac{f(t)}{t^k} = 0.$$
Also, throughout this paper, a constant $C$ is independent of $\eps$, but may differ from line to line. 

\section{Eigenvalue convergence:  A coarse estimate} \label{sec3}
In this section, we prove a coarse version of Theorem~\ref{thm1}.
Let $(M, g, \phi)$ be the compact Bakry-\'Emery manifold, with or without boundary, and let 
\begin{equation} \label{eq:Me} M_{\e} = \{ (x, y) \mid x \in M, \quad 0 \leq y \leq \e f(x) \} \subset M \times \R^+, \quad f(x) := e^{-\phi(x)}. \end{equation} 
Let $\{\mu_k \}_{k=0} ^{\infty}$ {and $\{ \psi_k \}_{k=0} ^{\infty}$} be {respectively} the eigenvalues {and eigenfunctions for} the drift Laplacian $\Delta$ on $M$ (if $\pa M\neq \emptyset$, we endow it the Neumann boundary condition), and let $\{ \mu_k (\e) \}_{k=0} ^{\infty}$ be the eigenvalues for $\tilde\Delta=\Delta + \pa_y ^2$ on $M_\e$ with corresponding 
{orthogonal} eigenfunctions $\{ \varphi_{j,\eps} \}_{j=0} ^{\infty}$.  We assume the eigenfunctions are normalized so that
$$ \int_{M_\eps} \varphi_{j,\eps} \varphi_{k, \eps} = \eps \delta_j ^k.$$
In particular, the volume of $M_\eps$ is $\eps$. This normalization depends only on $f$ and $M$.

We use $\nabla$ and $\Delta$ as the gradient and Laplace operators, respectively, of $M$, and $\tilde\nabla=(\nabla,\frac{\pa}{\pa t})$ and $\tilde{\Delta}$ as the gradient {and Laplace operators, respectively,} of $M_\eps\subset M\times \mathbb R ^+$.  

We prove the theorem by induction.  For $k=0$, the statement of Theorem~\ref{thm1} is trivial. We shall prove the theorem for $k\geq 1$, assuming that for $1,\cdots,k-1$, the theorem has been proven.  By a theorem of Uhlenbeck \cite{uh}, for generic manifold $\mu_1, \ldots, \mu_k$ are simple; that is, all eigenspaces with respect to the eigenvalues $\mu_1,\cdots,\mu_k$ are of multiplicity one.  Since the eigenvalues {are continuous with respect to continuous deformations of the domain,} it is sufficient to prove the theorem under this additional assumption.
 
\begin{lemma}\label{lemq1}
Using the above notation, $\mu_k (\epsilon) \leq \mu_k+O(\eps^2)$.
\end{lemma}

\begin{proof} Considering $\psi_k$ as functions on $M_\eps$, they are orthogonal with respect to the measure $dV_{g} dy$. By Proposition~\ref{prop4}, we have
\[
\sum_{j=0}^{k}\mu_j(\eps)\leq\sum_{j=0}^k \mu_j.
\]
By the inductive assumption, we have
\[
\mu_j\leq \mu_j(\eps)+O(\eps^2)
\]
for all $j<k$. The lemma follows from the above {two} inequalities.
\end{proof} 

For any $0 \leq r \leq \eps$, and for $0 \leq i \leq k$, let 
$$b_i (x,r) := \varphi_{i,\eps} (x, r f(x)) \textrm{ and }
A_k=\sum_{j=0}^k\int_{M_\eps}\left(\frac{\pa\varphi_{j,\eps}}{\pa y}\right)^2(x,y).$$ 

Intuitively, since the domain $M_\eps$ is very thin, the eigenfunctions $\vphi_{i, \eps}$ should be almost constant along the $y$ direction.  The following lemma confirms this intuition by quantifying the heuristic argument.   

\begin{lemma} \label{lemq2} Using the above notations, we have
\begin{equation} \label{eq:c-ortho}\left| \int_{M} b_i (x, r) b_j (x, r) f(x)  - \delta_i ^j \right| \leq C\sqrt{\eps A_k}\qquad \forall \, \, 0 \leq i,j \leq k\end{equation}
for all $0\leq r\leq \eps$.
\end{lemma}

\begin{proof}  \footnote{For simplicity of notation, we drop the subscript $\eps$ from $\varphi$.}  For any $0 \leq r \leq \eps$, $0 \leq y \leq \eps f(x)$, and $0 \leq i, j \leq k$, 
\begin{align}\label{eq:C1}
\begin{split}
&| b_i (x, r) b_j (x, r) - \varphi_i (x,y) \varphi_j (x,y) | \leq \int_0 ^{\epsilon f(x)} \left| \pa_y \left( \varphi_i (x,y) \varphi_j(x,y) \right) \right| dy \\
& \leq \int_0 ^{\epsilon f(x)} \left(\left|\frac{\pa \varphi_i}{\pa y}\right|\cdot |\varphi_j| +\left |\frac{\pa\varphi_j}{\pa y}\right| \cdot|\varphi_i|\right) (x, y)dy.
\end{split}
\end{align}
Note that for any $0\leq r\leq \eps$,
$$ \eps \int_M b_i (x, r) b_j (x, r) f(x)  = \int_0 ^{\eps f(x)} \int_{M} b_i (x, r) b_j (x, r )   = \int_{M_{\eps}} b_i (x, r) b_j (x, r) , $$
and
$$\int_{M_{\eps}} \varphi_i(x, y) \varphi_j (x,y)  = \eps \delta_i ^j .$$
Then
$$\left|\eps  \int_{M} b_i (x, r) b_j (x,r) f(x) - \eps \delta_i ^j \right| = \left| \int_{M_{\eps}} (b_i (x, r) b_j (x, r) - \varphi_i (x,y) \varphi_j (x,y) )  \right|,$$
which by (\ref{eq:C1}), 
\begin{align} \label{eq:subst} 
\begin{split}& \leq  \int_{M_{\eps}} \int_0 ^{\eps f(x)}\left(\left|\frac{\pa \varphi_i}{\pa y}\right|\cdot |\varphi_j| +\left |\frac{\pa\varphi_j}{\pa y}\right| \cdot|\varphi_i|\right)(x,t), \\&
\leq\eps||f||_\infty \int_{M_\eps}\left(\left|\frac{\pa \varphi_i}{\pa y}\right|\cdot |\varphi_j| +\left |\frac{\pa\varphi_j}{\pa y}\right| \cdot|\varphi_i|\right)\\
& \leq \eps || f ||_{\infty} \left( \sqrt{A_k}\cdot || \varphi_j ||_{L^2(M_\eps)} + \sqrt{A_k} \cdot ||\varphi_i ||_{L^2(M_\eps)}  \right){.} \end{split}
\end{align}
Since  $|| \varphi_i ||_{L^2(M_\eps)}  = \sqrt{ \eps}$, 
$$ \left|\eps  \int_{M} b_i (x, r) b_j (x,r) f(x) - \eps \delta_i ^j \right| \leq C \eps^{3/2}\sqrt{A_k}.$$ 
\end{proof} 

\begin{cor} Using the same notations as above, we have
\[
\left| \int_{M} b_i (x, r) b_j (x, r) f(x)  - \delta_i ^j \right| \leq C \eps\qquad \forall \, \, 0 \leq i,j \leq k, 0\leq r\leq \eps.
\]
\end{cor}

\begin{proof} This follows from the fact that {$A_k \leq \sum_{j=1} ^k || \tilde{\n} \vphi_j ||_{L^2(M_\eps)} ^2 = \sum_{j=1} ^k \mu_j (\eps)\eps \leq C \eps$}, where the constant $C$ depends only on $k$ (not on $\eps$).
\end{proof} 

The following result is a  coarse version of Theorem~\ref{thm1}.
\begin{lemma}\label{lemq3}
Under the same condition as in Theorem~\ref{thm1},  and assuming that Theorem~\ref{thm1} is true for $j<k$, we have
\[
\mu_k\leq\mu_k(\eps)+C(\eps^2+\sqrt{\eps A_k}).
\]
In particular, this estimate and Lemma~\ref{lemq1} imply that $\mu_k (\eps) = \mu_k + O(\eps)$ for all $k \geq 0$.  
\end{lemma}

\begin{proof}  
Define inductively that $\tilde b_0(x,r)=b_0(x,r)$,
\[
\tilde b_k(x,r)=b_k(x,r)+\sum_{j=0}^{k-1} c_{kj}(r) b_j(x,r),
\]
where for any $k \geq 0$, $c_{kj} (r)$ are functions of $r$ such that
\[
\tilde b_k\perp b_1,\cdots, b_{k-1}
\]
with respect to the measure ${f} dV_g$.  {The intuition behind defining the functions $\tilde b_k$ is that the functions $b_k$ are {\it almost orthogonal,} as proven in Lemma ~\ref{lemq2}, but to obtain the estimate $\mu_j(\eps) +O(\eps^2) \leq \mu_j$ we need {\it orthogonal} test functions to apply Proposition ~\ref{prop4}.  In the arguments below, we first apply Proposition ~\ref{prop4} and then show that the modified functions $\tilde b_k$ are very close to the original functions $b_k$, thereby obtaining the required estimate.}  

By Proposition~\ref{prop4}, we have
\begin{equation} \label{eq:inf1} 
\sum_{j=0}^{k}  \mu_{{j}} \leq \sum_{j=0}^{k}\frac{\int_M|\n\tilde b_j|^2 f}
{\int_M\tilde b_j^2 f}.
\end{equation}
By Lemma~\ref{lemq2}, $c_{kj}(r) \sim O(\sqrt{\eps A_k})$ uniformly for $0 < r \leq \eps$. Thus by the definition of $\tilde b_k$, using Lemma~\ref{lemq2} again, we have
\[
{\sum_{j=0} ^k \mu_j \leq} \sum_{j=0}^{k}\frac{\int_M|\n\tilde b_j|^2 f}{\int_M\tilde b_j^2 f}
\leq (1+C\sqrt{\eps A_k})\sum_{j=0}^k \int_M|\n b_j|^2 f.
\]
Since {the above inequality} holds for all $r$, integrating from $0$ to $\eps$, we have
\begin{equation} \label{eq:inf2}
 \eps  \sum_{j=0}^{k}\mu_j \leq (1+C\sqrt{\eps A_k})\sum_{j=0}^{k}\int_0 ^{\eps} \int_{M} {|\n b_{j}|^2 f} .
\end{equation}
We compute
$$\n b_j (x, r) = (\n \varphi_j ) (x, r f(x)) + r  \frac{\pa \varphi_j}{\pa y} (x, r f(x) ) \n f (x).$$
Using the Cauchy inequality, we get
\begin{eqnarray} \label{eq:nvarphi} 
&|\n b_j (x, r) |^2 \leq (1+\eps^2|\n f|^2)|\tilde{\nabla} \varphi_{j,\eps}|^2.
\end{eqnarray} 

Therefore,
$$\int_0 ^{\eps} \int_{M} |\n b_j (x,r)|^2 f(x)  \leq (1 + C \eps^2) \int_{M_\eps} |\tilde{\n} \varphi_{j,\eps} |^2 = (1 + C \eps^2) \mu_{j}( \eps) \eps$$
for $j\leq k$.

The above  estimate together with {(\ref{eq:nvarphi})} show that 
$$\eps\sum_{j=0}^{k} \mu_j \leq  \eps (1+C\sqrt{\eps A_k})(1+C\eps^2)\sum_{j=0}^{k}\mu_{j}(\eps).$$

Dividing by $\eps$ and letting $\eps \to 0$, this estimate together with an induction argument completes the proof of the lemma.  The precise estimate $\mu_k (\eps) = \mu_k + O(\eps^2)$ to complete the proof of Theorem ~\ref{thm1} will be demonstrated in the final section.  
\end{proof}  

{\bf Proof of Corollary 1.}
The corollary follows immediately from Lemma~\ref{lemq3} and Proposition~\ref{prop2}. 
\qed

\section{A maximum principle}
The Neumann eigenvalues are continuous functions with respect to the manifold $M$.  Therefore, to estimate the eigenvalues, we may use an exhaustion of $M$, 
$$M^{\delta} = \{ x \in M \mid \textrm{dist}(x, \pa M) \geq \delta\}, \quad \delta > 0.$$
On $M^{\delta}$, $f$ has a positive lower bound.  Thus using the variational principle, we may, without loss of generality assume that $f$ is not only positive but is a constant in a neighborhood of $\pa M^{\delta}$.  For the rest of the paper we make such an assumption. 

The usual maximum principle for the gradient estimate is {the following.}  
{L}et
\[
H=\frac 12|\n\varphi|^2 +F(\varphi),
\]
where $F$ is a smooth function of one variable, and let $x_0$ be an interior point of $M$ {at which $H$ reaches its maximum}. Then  
\[
0\geq|\nabla^2\varphi|^2+\n\varphi\n(\Delta\varphi)+{{\rm Ric}}(\n\varphi,\n\varphi)+F'(\varphi)\Delta\varphi+F''(\varphi)|\n\varphi|^2.
\]

The above inequality is very useful for obtaining lower bounds on  the first eigenvalue of a Laplace or Schr\"odinger operator;  for more details, we refer to  the book~\cite{S-Yau}.

However, it is not appropriate to apply the above maximum principle directly to the manifold $M_\eps$ for the following reasons:
\begin{enumerate}
\item $M_\eps$ need not be convex, even  if $M$ is. As we know, 
if $M$ is convex, the maximum of $H$ must be reached in the interior of $M$. In general, we don't have such a property for $M_\eps$.
\item The natural Ricci curvature attached to the problem is {${\rm Ric}_\infty$}, not the Ricci curvature of $M_\eps$, which is essentially {${\rm Ric}$.}
\end{enumerate}

\begin{remark}  
The choice of $F$ is highly technical. In the Li-Yau's case~\cite{li-yau}, which is the simplest case, $F(x)=\frac 12 x^2$. In Zhong-Yang's case~\cite{zhong-yang}, $F$ is (up to a constant)
\[
F(x)=1-x^2+a\left(\frac{4}{\pi}(\arcsin x+x\sqrt{1-x^2}-2x)\right),
\]
where $a$ is a positive constant. More sophisticated choices of $F$ can be found in~\cite{ling} and ~\cite{ling-lu}.   
\end{remark}

As in the previous section{s}, we assume $M$ is a compact manifold with or without boundary. Let $U$ be an open set of $M$ and let $(x_1,\cdots, x_n)$ be  a local coordinates system on $U$.  Let $\varphi_{k,\eps}$  be the  Neumann eigenfunctions of the eigenvalues $\mu_k(\eps)$ with the $L^2$ norm normalized to be ${\sqrt\eps}$.  We let
\[
\psi(x)=\varphi_{k,\eps}(x,0),\quad x\in M.
\]
The technical heart of {this paper} is Theorem~\ref{thm2}, which 
 implies the  following key results of this section.  

\begin{lemma}\label{lem23}
With the above notation, as $\eps \to 0$,
\begin{align}
&\frac{\pa^2 {\vphi_{k,\eps}}}{\pa y^2}-\nabla\log f(x)\nabla \psi=o(1), \label{uji} \\
&(\n \psi, \n\frac{\pa^2{\vphi_{k,\eps}}}{\pa y^2})-(\nabla ^2\log f)(\nabla \psi,\nabla \psi)-\n^2\psi(\n\psi,\n\log f)
=o(1).\label{uji-2}
\end{align}
\end{lemma}

\begin{proof}
Since $\varphi_{k,\eps}$ satisfies the Neumann condition, we have
\begin{align}\label{q-1}
\begin{split}
&\frac{\pa \varphi_{k,\eps}}{\pa y}(x,0)=0;\\
&\frac{\pa \varphi_{k,\eps}}{\pa y}(x, \eps f(x))-\eps\nabla f(x)\nabla \vphi_{k,\eps}(x, \eps f(x))=0,
\end{split}
\end{align} 
for any $x \in M$.  {Applying the mean value theorem to the above equations,} we have
\[
\eps f(x)\left (\frac{\pa^2 \varphi_{k,\eps}}{\pa y^2}(x,\xi(x))-\nabla\log f(x)\nabla \varphi_{k,\eps}(x,\eps f(x))\right)=0,
\]
where $\xi(x) \in (0, \eps f(x))$.  Theorem ~\ref{thm2} then implies ~\eqref{uji}.  

Assume now that at $x$, the local coordinate system is normal. 
For the second statement,  taking partial derivatives with respect to $x_j$ in the second equation of~\eqref{q-1} gives
\begin{align*}
&\frac{\pa^2 \varphi_{k,\eps}}{\pa x_j \pa y}(x,\eps f(x))+\eps\frac{\pa^2 \varphi_{k,\eps}}{\pa y^2}(x,\eps f(x))\frac{\pa f}{\pa x_j}\\
&-\eps\frac{\pa^2 f(x)}{\pa x_i\pa x_j}\frac{\pa \varphi_{k,\eps}}{\pa x_i}(x,\eps f(x))-\eps \frac{\pa f(x)}{\pa x_i}\frac{\pa^2 \varphi_{k,\eps}}{\pa x_i\pa x_j}(x,\eps f(x))\\&
-\eps^2\frac{\pa f}{\pa x_i}\frac{\pa^2 \varphi_{k,\eps}}{\pa x_i\pa y}(x,\eps f(x))\frac{\pa f}{\pa x_j}=0.
\end{align*}
Since $\pa\varphi_{k,\eps}/\pa y=0$ on $\{y=0\}$, we have
\[
\frac{\pa^2 \varphi_{k,\eps}}{\pa x_j\pa y}(x,0)=0, \quad \frac{\pa^3 \varphi_{k,\eps}}{\pa x_j\pa x_i\pa y}(x,0)=0.
\]
The mean value theorem implies
$$\frac{\pa^2 \varphi_{k,\eps}}{\pa x_j \pa y}(x,\eps f(x)) = \eps f(x) \frac {\pa^3 \varphi_{k,\eps}}{\pa x_j\pa y ^2}(x,\xi(x))$$
for some $\xi(x) \in (0, \eps f(x))$. Using Theorem~\ref{thm2} again, we have 
\begin{align*}
&\eps f(x)\frac{\pa^3 \varphi_{k,\eps}}{\pa x_j \pa y^2}(x,0)
+\eps\frac{\pa^2 \varphi_{k,\eps}}{\pa y^2}(x,0)\frac{\pa f}{\pa x_j}-\eps
\frac{\pa^2 f(x)}{\pa x_i\pa x_j}\frac{\pa  \varphi_{k,\eps}}{\pa x_i}(x,0)\\
&-\eps \frac{\pa f(x)}{\pa x_i}\frac{\pa^2  \varphi_{k,\eps}}{\pa x_i\pa x_j}(x,0)
-\eps^2\frac{\pa f}{\pa x_i}\frac{\pa^2 \varphi_{k,\eps}}{\pa x_i\pa y}(x,0)\frac{\pa f}{\pa x_j}=o(\eps f(x))
\end{align*}
{as $\eps \to 0$.}  
Thus we have
\begin{align*}
&
\frac{\pa^3 \varphi_{k,\eps}}{\pa x_j \pa y^2}(x,0)+\frac{\pa^2 \varphi_{k,\eps}}{\pa y^2}(x,0)\frac{\pa\log  f}{\pa x_j}-\frac{1}{f}
\frac{\pa^2 f(x)}{\pa x_i\pa x_j}\frac{\pa \varphi_{k,\eps}}{\pa x_i}(x,0)\\&-\frac{\pa \log f(x)}{\pa x_i}\frac{\pa^2 \varphi_{k,\eps}}{\pa x_i\pa x_j}(x,0)
=o(1).
\end{align*}
Using ~\eqref{uji}, we get
\begin{align*}
&
\frac{\pa^3 \varphi_{k,\eps}}{\pa x_j \pa y^2}(x,0) + \n \log f(x) \n {\psi (x) } \frac{\pa \log f{(x)}}{\pa x_j}  -\frac{1}{f{(x)}}
\frac{\pa^2 f(x)}{\pa x_i\pa x_j}\frac{\pa { \psi(x)}}{\pa x_i }\\&-\frac{\pa \log f(x)}{\pa x_i}\frac{\pa^2 \psi}{\pa x_i\pa x_j}({x}) = o(1),
\end{align*} 
which implies ~\eqref{uji-2}.
\end{proof} 

For our new maximum principle, we consider  
\[ 
H(x,y)=\frac 12 |\tilde \nabla \varphi_{k,\eps}|^2+F(\varphi_{k,\eps}),
\]
where $F$ is a smooth function of one variable.  
Assume that $(x_0,0)$ is the point at which {$H$} reaches the maximum on $\{y=0\}$,
where $x_0$ is in the interior of $M$.

At $(x_0,0)$, {we have }
$$ \nabla H(x_0,0)=0 \textrm{ and } \Delta H(x_0,0)\leq 0.$$
The difficulty is that $H$ satisfies {an} elliptic equation {with respect to} $\tilde{\Delta}$, rather than $\Delta$.  To obtain the new maximum principle, we need to estimate the second derivative of $H$ in the $y$-direction.  

\begin{lemma}\label{l55}
At $(x_0, 0)$, 
\[
\frac{\pa^2 H}{\pa y^2}= \nabla^2\log f (\nabla \psi,\nabla \psi)+\left(\frac{\pa^2 \vphi_{k,\eps}}{\pa y^2}\right)^2+o(1){, \quad \textrm{as} \quad \epsilon \to 0}.
\]
\end{lemma}

\begin{proof} Using the normal coordinates at $x_0$, 
we have
\[
\frac{\pa H}{\pa y}={\sum_{i=1} ^n  \frac{\pa \vphi_{k,\eps}}{\pa x_i}\frac{\pa^2 \vphi_{k,\eps}}{\pa x_i \pa y}  +{\sum_{i=1} ^n}\frac{\pa \vphi_{k,\eps}}{\pa y}\frac{\pa^2 \vphi_{k,\eps}}{\pa y^2}+F'(\vphi_{k,\eps})\frac{\pa \vphi_{k,\eps}}{\pa y}},
\]
and
\begin{align*}&
\frac{\pa^2 H}{\pa y^2}= \sum_{i=1}^n \left(\frac{\pa^2 \vphi_{k,\eps}}{\pa x_i \pa y} \right)^2 + \sum_{i=1} ^n\frac{\pa \vphi_{k,\eps}}{\pa x_i} \frac{\pa^3 \vphi_{k,\eps}}{\pa x_i \pa y^2}\\
&+ \left( \frac{\pa^2 \vphi_{k,\eps}}{\pa y^2} \right)^2 + \frac{\pa \vphi_{k,\eps}}{\pa y} \frac{\pa^3 \vphi_{k,\eps}}{\pa y^3} + F''(\vphi_{k,\eps})  \left( \frac{\pa \vphi_{k,\eps}}{\pa y} \right)^2 + F'(\vphi_{k,\eps}) \frac{\pa^2 \vphi_{k,\eps}}{\pa y^2}.
\end{align*}
Since $\vphi_{k,\eps}$ satisfies the Neumann boundary condition, $\frac{\pa \vphi_{k,\eps}}{\pa y}$ {and $\frac{\pa^2 \vphi_{k, \eps}}{\pa x_i \pa y}$ vanish} on $\{ y=0 \}$.
Thus we have
\[
\frac{\pa^2 H}{\pa y^2}=  { \sum_{i=1} ^n}  \frac{\pa \vphi_{k,\eps}}{\pa x_i} \frac{\pa^3 \vphi_{k,\eps}}{\pa x_i \pa y^2}
+ \left( \frac{\pa^2 \vphi_{k,\eps}}{\pa y^2} \right)^2 + F'(\vphi_{k,\eps}) \frac{\pa^2 \vphi_{k,\eps}}{\pa y^2}.
\]
Using Lemma ~\ref{lem23}, we have 
$$
\frac{\pa^2 H}{\pa y^2}=\n^2\log f(\n\psi,\n\psi)+{\n ^2 \psi (\n \psi, \n \log f) + \left( \frac{\pa^2 \vphi_{k,\eps}}{\pa y^2} \right)^2 }
+F'(\psi)\frac{\pa^2 {\varphi_{k,\eps}}}{\pa y^2}+o(1).
$$
Since at $x_0$, $\n H=0$, we have
\[
{\sum_{j=1} ^n} \frac{\pa^2 \psi}{\pa x_i\pa x_j} \frac{\pa \psi}{\pa x_j}+F'(\psi)\frac{\pa\psi}{\pa x_i}=0, \quad  \textrm{for each} \quad 1 \leq i \leq n.
\]
Using the above equality and Lemma~\ref{lem23}, the second and fourth terms on the right side of the expression for $\frac{\pa^2 H}{\pa y^2}$ cancel.
\end{proof} 

\begin{theorem}[Maximum Principle]
With the above notations, we have at $(x_0, 0)$ 
\[
o(1)\geq|\nabla^2\psi|^2+\n\psi\n(\tilde\Delta\varphi_{k,\eps})+{\rm Ric}_\infty(\n\psi,\n\psi)+F'(\psi){\tilde \Delta} \varphi_{k,\eps}+F''(\psi)|\n\psi|^2.
\]
\end{theorem}

\begin{proof} 
By the Bochner formula, we have
\begin{align*}
&\tilde \Delta H=|\tilde\nabla ^2\varphi_{k,\eps}|^2+{\rm Ric} (\tilde\n\varphi_{k,\eps},\tilde\n\varphi_{k,\eps})+\tilde\n\varphi_{k,\eps}\tilde\n(\tilde\Delta\varphi_{k,\eps})\\
&+F'(\varphi_{k,\eps})\tilde\Delta\varphi_{k,\eps}+F''(\varphi_{k,\eps})|\tilde\n\varphi_{k,\eps}|^2.
\end{align*}
On $\{y=0\}$, we have
\begin{align*}
& |\tilde\n^2\varphi_{k,\eps}|^2=|\n^2\psi|^2+\left( \frac{\pa^2 \varphi_{k,\eps}}{\pa y^2}\right)^2,\\
&{\rm Ric} (\tilde\n\varphi_{k,\eps},\tilde\n\varphi_{k,\eps})={\rm Ric}(\n\psi,\n\psi),\\
&\tilde\Delta\varphi_{k,\eps}=\Delta\psi+\frac{\pa^2\vphi_{k,\eps}}{\pa y^2},
\end{align*}
Thus we have 
\begin{align*}
&
\tilde\Delta H=|\n^2\psi|^2+\left(\frac{\pa^2 {\vphi_{k,\eps}} }{\pa y^2}\right)^2+{\rm Ric}(\n\psi,\n\psi)\\
&+\nabla\psi\nabla(\tilde\Delta\varphi_{k,\eps})+F'(\psi)\tilde\Delta\varphi_{k,\eps}+F''(\psi)|\n\psi|^2.
\end{align*}
Using Lemma~\ref{l55}, noting that at $(x_0,0)$
\[
\tilde\Delta H=\Delta H+\frac{\pa^2 H}{\pa y^2}\leq \frac{\pa^2 H}{\pa y^2},
\]
completes the proof.
\end{proof} 
\subsection{Applications}  
Our work not only has applications to Bakry-\'Emery geometry but also to Ricci solitons.  We recall the main result of Futaki and Sano \cite{futaki}.
\begin{theorem}[Futaki-Sano]  Let $M^n$ be a compact smooth manifold of dimension at least $4$.  If $g$ is a non-trivial gradient shrinking Ricci soliton on $M$ (see definition 1.1 of \cite{futaki}), then the diameter of $M$ with respect to $g$ is bounded below by $\frac{10 \pi}{13 \sqrt{\gamma}}$, where $\gamma$ is a constant determined by $g$.  
\end{theorem} 

This result is proven by using Ling's gradient estimates \cite{ling} to demonstrate a lower bound for the first non-zero eigenvalue of a certain Bakry-\'Emery Laplacian.  Our {\bf Principle } shows that one may directly apply Ling's estimates to the Bakry-\'Emery Laplacian to obtain the result.  It is reasonable to expect that one may similarly express elliptic geometric equations, like the Ricci soliton equation, in terms of a Bakry-\'Emery Laplacian and exploit the eigenvalue estimates from Riemannian geometry together with our {\bf Principle} to produce interesting results.  

Another application arises from the so-called \em fundamental gap: \em the difference between the first two Dirichlet eigenvalues of a domain in $\R^n$.  Andrews and Clutterbuck \cite{ac} recently demonstrated an optimal lower bound of $3 \pi^2 / d^2$ for the fundamental gap of any convex domain in $\R^n$ with diameter $d$.  By Proposition~\ref{prop2}, the fundamental gap can be interpreted as the first Neumann eigenvalue on certain Bakry-\'Emery manifold, and in particular, techniques of ~\cite{andrews}, ~\cite{ac} together with our work imply the following.  

\begin{theorem} Let $\Omega \subset \R^n$ be a convex domain with piecewise smooth boundary and diameter $d$.  {Let $f \in \cC^2 (\bar\Omega)$.}  If $f$ satisfies 
$$\left( \n \log f(y) - \n \log f(x) \right) \cdot \frac{y-x}{|y-x|} \geq \frac{  4 \pi}{d} \tan \left( \frac{ \pi |y-x|}{d} \right) \quad \forall \quad x \neq y \textrm{ in } \Omega,$$
then the first non-trivial Neumann eigenvalue of the Bakry-\'Emery Laplacian with respect to the weight function $\phi = - \log (f^2)$ is bounded below by $3 \pi^2 / d^2$.  Moreover, the first Neumann eigenfunction for the Euclidean Laplacian on  
$$\Omega_{\epsilon} := \left\{ (x, y) \mid x \in \Omega, \, 0 \leq y \leq \epsilon f^2 (x) \right\} \subset \R^{n+1}$$
is bounded below by $3\pi^2 / d^2 - C \epsilon^2$, where $C$ is a fixed constant that depends only on $n$ and $\Omega$.
\end{theorem}

\qed
\section{The approximation of eigenfunctions}\label{sec5}
It is not hard to write down the eigenfunctions formally. Let $\varphi$ be a Neumann  eigenfunction of $M_\eps$ with eigenvalue $\lambda$.  Write 
\[
\varphi=\sum_{k=0}^\infty y^k\varphi_k,
\]
where $\varphi_k$ are functions on $M$. Then we have (formally) 
\[
\Delta\varphi_k+\lambda\varphi_k+(k+1)(k+2)\varphi_{k+2}=0
\]
for all $k\geq 0$. Since $\pa\vphi/\pa y=0$ on $\{ y = 0\}$,  we have $\varphi_1=0$ and hence $\varphi_{2k+1}=0$ for all $k$. Let 
\[
H\varphi=-\Delta\varphi-\lambda\varphi.
\]
Then
\[
\varphi_{2k+2}=\frac{H\varphi_{2k}}{(2k+1)(2k+2)}=\frac{H^{k+1}\vphi_0}{(2k+2)!}.
\]
Formally, we have 
\[
\varphi=\sum_{k=0}^\infty\frac{y^{2k} H^k}{(2k)!} \vphi_0 =\cosh ({y}\sqrt{H})\varphi_0.
\]
The differential equation for $\varphi_0$ follows from the Neumann boundary condition
\[
\sqrt H\sinh\left({\eps f(x) \sqrt{ H}}\right)\vphi_0-\eps\left.\n f\cdot\n\left(\cosh ({y}\sqrt{ H})\vphi_0\right|_{y=\eps f(x)}\right)=0.
\]

We are not able to prove  the full regularity of the above equation at this moment.  But a partial solution, namely, a good approximation to  the eigenfunctions, is enough for our application.  Very roughly speaking,  in this section, we prove
\[
\vphi=\vphi_0+y^2\vphi_2 +O(\eps^3).
\]

To state our results precisely, we recall the global Schauder estimates~\cite{GT}*{Theorem 6.6, Theorem 6.30} and the interpolation inequalities. 

We let
\begin{align*}
&B_I=\{(x,y)\in M_\eps\mid y=0\};\\
&B_{II}=\{(x,y)\in M_\eps\mid y=\eps f(x)\};\\
&B_{III}=\{(x,y)\in M_\eps\mid x\in\pa M\}.
\end{align*}
Then $B_I\cup B_{II}\cup B_{III}=\pa M_\eps$.

Let 
$$ u_1=u|_{\pa M_\eps} \textrm{ and } u_2=\left.\frac{\pa u}{\pa n}\right|_{\pa M_\eps} \text{ on the smooth part of $\pa M_\eps$.}$$ 
Define the weighted H\"older norm by
\[
||u||_{C^{k,\alpha}_\eps}=\eps^{k+\alpha}[u]_{C^{k,\alpha}}+\cdots+\eps^\alpha[u]_{C^\alpha}+||u||_{C^0},
\]
and 
$||u||_{C^\alpha_\eps}=||u||_{C^{0,\alpha}_\eps}$, where $[\quad]_{C^{k,\alpha}}$ are the standard notation{s} defined in ~\cite{GT}.  Using these weighted norms, the constants in the {Schauder estimates on $M_\epsilon$} are independent of $\epsilon$.  
Let $0<\alpha<1$.  Let $L$ be a second order uniform elliptic operator with $C^\alpha$-bounded coefficients. Then we have the following version of global Schauder estimates on $M_\eps$
\begin{align}\label{s-1}
\begin{split}
&||u||_{C^{2,\alpha}_\eps}
\leq C(||u||_{C^0}+\eps^2||Lu||_{C^\alpha_\eps}+\eps||u_2||_{C^{1,\alpha}_\eps}),
\end{split}
\end{align}
and
\begin{align}\label{s-2}
\begin{split}
&||u||_{C^{2,\alpha}_\eps}
\leq C(||u||_{C^0}+\eps^{2}||Lu||_{C_\eps^\alpha}+\eps|| u_2||_{C_\eps^{1,\alpha}(B_{III})}
+||u_1||_{C_\eps^{2,\alpha}(B_I\cup B_{II})}).
\end{split}
\end{align}

The Sobolev inequality on $M_\eps$ is
\begin{equation}\label{kl-1}
\left(\int_{M_\eps}|u|^{2\frac{n+1}{n-1}}\right)^{\frac{n-1}{n+1}}\leq C\eps^{-\frac{2}{n+1}} (\int_{M_\eps}|\n u|^2+\int_{M_\eps}|u|^{2}).
\end{equation}

Define the H\"older norm in the $y$-direction to be 
\[
[u]_{C^\alpha_y}=\max_{x\in M}\sup_{0\leq y_1,y_2\leq\eps f(x)}\frac{|u(x,y_1)-u(x,y_2)|}{|y_1-y_2|^\alpha}.
\]
Then we have the following. 

\begin{theorem}\label{thm2}
For $M_{\eps}$ defined by (\ref{eq:Me}) such that $f(x) = e^{- \phi(x)}$ is constant in a neighborhood of $\pa M$, the Neumann eigenfunctions $\vphi_{k, \eps}$ of $M_{\eps}$ satisfy 
\begin{align*}
&\nabla \vphi_{k,\eps}=O(1),\\
&\left|\left|\frac{\pa^2 \varphi_{k,\eps}}{\pa y^2}\right|\right|_{C^\alpha _y}+\left|\left|\frac{\pa^2 \varphi_{k,\eps}}{\pa x_j\pa y}\right|\right|_{C^\alpha _y}=O(1),\qquad 1\le j\le n\\
&\left|\left|\frac{\pa^3 \varphi_{k,\eps}}{\pa x_j\pa y^2}\right|\right|_{C^\alpha _y}+
\left|\left|\frac{\pa^3 \varphi_{k,\eps}}{\pa x_i\pa x_j\pa y}\right|\right|_{C^\alpha _y}=O(1),\qquad 1\le i, j\leq n
\end{align*}
for any {$0< \alpha <1$}, where $(x_1,\cdots,x_{n})$ {is} any local coordinate system of $M$.  
\end{theorem}

As defined in \S\ref{sec3}, let $\{\mu_k \}_{k=0} ^{\infty}$ {and $\{ \psi_k \}_{k=0} ^{\infty}$} be respectively the eigenvalues and eigenfunctions for the drift Laplacian. Moreover, we assume that $\{ \psi_k \}_{k=0} ^{\infty}$ is an orthonormal basis {for} the space of $L^2$ functions of $M$  {with respect to the weighted measure $e^{-\phi} dV_g$.}

Define the functions
\[
\eta_k := -\nabla  \phi\nabla\psi_k, 
\]
{on $M$, and let
$$ U_k := \psi_k+\frac 12y^2\eta_k$$
be functions on $M_\epsilon$.}  By our definition of $f(x) = e^{- \phi(x)}$, $\eta_k$ are smooth up to the boundary.
Note that since $\psi_k$ is a Bakry-\'Emery eigenfunction, it satisfies 
\begin{equation} \label{eq:be-psik} \Delta \psi_k + \eta_k = - \mu_k \psi_k. \end{equation}
Since $\psi_k$ and $\eta_k$ are independent of $y$, 
\begin{equation}\label{eq:3}
\tilde \Delta U_k = -\mu_k \psi_k + \frac 12y^2\Delta\eta_k = - \mu_k U_k + \frac{1}{2} y^2 ( \Delta \eta_k + \mu_k \eta_k). 
\end{equation} 

We compute directly, 
\begin{equation} \label{eq:4}
\left.\frac{\pa U_k}{\pa n}\right|_{\pa M_\eps}=\left\{
\begin{array}{ll}
0& {\rm on } \,\,B_I\cup B_{{III}};\\
-  \frac{\displaystyle \eps^3f^2 \nabla f\nabla\eta_k}{\displaystyle  2(1+\eps^2|\nabla f|^2)^{1/2}}& {\rm on }\,\, B_{{II}}.
\end{array}
\right. \end{equation} 

Define
\[
w_{k,\eps}=\alpha_{k,k}\varphi_{k,\eps}+U_k+\sum_{j=0}^{k-1}\alpha_{k,j}\varphi_{j,\eps},
\]
where $\alpha_{k,j}$ are  defined such that $w_{k,\eps}\perp\varphi_{0,\eps}, \cdots, \varphi_{k,\eps}$ in $\cL^2 (M_{\eps})$.  

{Heuristically, by ~\eqref{eq:3} and~\eqref{eq:4}, $U_k$ should be very close to the eigenfunction $\varphi_{k,\eps}$, and the function $w_{k, \eps}$ should be approximately the difference between $U_k$ and $\vphi_{k, \eps}$.  In what follows, we prove that $w_{k, \eps}$ and its $y$ derivative (denoted $r_{k, \eps}$ in ~\eqref{sk-1-1} below) are indeed small, thereby rigorously demonstrating our intuitive heuristic.}

The following  inequality will be used repeatedly in the rest of the paper: let $\vphi$ be  a function on $M_\eps$. Then for any $p>0$, we have
\begin{align}
& \int_{B_{II}}|\vphi|^p\leq C\left(\frac 1\eps\int_{M_\eps}|\vphi|^p+p\sqrt{\int_{M_\eps}|\vphi|^{p-1}}\cdot
\sqrt{\int_{M_\eps}|\vphi|^{p-1}|\n\vphi|^2}\right).\label{5-7}
\end{align}

To prove~\eqref{5-7}, we observe that for any $0\leq y\leq\eps f(x)$, we have
\[
|\vphi|^p(x,\eps f(x))\leq|\vphi|^p(x,y)+p\int_0^{\eps f(x)}|\vphi|^{p-1}|\n\vphi| dy.
\]
Integrating over $M_\eps$ to both sides of the above equation and using the Cauchy-Schwarz inequality implies (\ref{5-7}).  
 
\begin{lemma}\label{lemm1}
For all $j < k$, $\alpha_{k,j}=O(\eps^2)$, and $\alpha_{k,k}=O(1)$.
\end{lemma}
\begin{proof} Integrating  by parts, {(\ref{eq:3}) and (\ref{eq:4})} give
\begin{align*}
&
\mu_{j} {(\eps)} \int_{M_\eps} \varphi_{j, \eps} U_k = -\int_{M_\eps} \tilde{\Delta} \varphi_{j, \eps} U_k = \int_{\pa M_\eps} \varphi_{j, \eps} \frac{\pa U_k}{\pa n} - \int_{M_\eps} \varphi_{j, \eps} \tilde{\Delta} U_k \\
&=\int_{\pa M_\eps} \varphi_{j, \eps} \frac{\pa U_k}{\pa n} +\mu_k \int_{M_\eps} \varphi_{j, \eps}  U_k - \frac 12\int_{M_\eps} \varphi_{j, \eps} y^2( \Delta \eta_k+\mu_k\eta_k).
\end{align*}
Thus by~\eqref{eq:4} again, we have 
\begin{equation} \label{eq:mje}
(\mu_{j}(\eps) - \mu_k) \int_{M_\eps} \varphi_{j, \eps} U_k  = {-} \int_{B_{II}} \varphi_{j, \eps} \frac{\pa U_k}{\pa n} - \frac 12\int_{M_\eps} \varphi_{j, \eps} y^2 (\Delta\eta_k+\mu_k\eta_k). 
\end{equation} 
We clearly have 
\begin{equation} \label{eq:11} \left| \int_{M_\eps} \varphi_{j, \eps} y^2 (\Delta\eta_k+\mu_k\eta_k) \right| \leq
C\sqrt{\int_{M_\eps} y^4}\cdot\sqrt{\int_{M_\eps}|\vphi_{j,\eps}|^2}
= O(\eps^3).\end{equation} 
Using~\eqref{5-7} for $p=1$, we have
$$\left|  \int_{B_{II}} \varphi_{j, \eps} \frac{\pa U_k}{\pa n} \right|
\leq C\eps^3.$$

By the generic assumption of the manifold  $M$ and Lemma {~\ref{lemq3}}, $\mu_j (\eps) = \mu_j + O(\eps) < \mu_k$ for all $j < k$ for $\eps$ sufficiently small.  Thus, dividing by $(\mu_{j}( \eps) - \mu_k)$ in (\ref{eq:mje}) gives 
\begin{equation} \label{eq:alphajk} \int_{M_\eps} \varphi_{j, \eps} U_k = O(\eps^3) \implies \alpha_{k,j} = O(\eps^2). \end{equation} 
That $\alpha_{k,k}$ is bounded follows from its definition.
\end{proof} 

A straightforward computation gives 
$$F_1 := \tilde\Delta w_{k,\eps}+\mu_k(\eps)w_{k,\eps} $$
\begin{equation} \label{sk-1} 
=(\mu_k(\eps)-\mu_k)U_k+\frac 12y^2(\Delta\eta_k+\mu_k\eta_k)+\sum_{j=0}^{k-1}\alpha_{k,j}(\mu_k(\eps)-\mu_j(\eps))\varphi_{j,\eps}
\end{equation} 
{with the boundary conditions} 
\begin{equation}\label{sk-1-1}
\frac{\pa w_{k,\eps}}{\pa n}=
\left\{
\begin{array}{ll}
0& \text{ on } B_I\cup B_{III}\\
\frac{\pa U_k}{\pa n}=O(\eps^3) &\text{ on } B_{II}
\end{array}
\right..
\end{equation}
Let 
$$r_{k, \eps} = \frac{\pa w_{k,\eps}}{\pa y}.$$
Then 
$$F_2 :=  \tilde\Delta r_{k,\eps}+\mu_k(\eps)r_{k,\eps} $$
\begin{equation}\label{sk-2}
=(\mu_k(\eps)-\mu_k)y\eta_k+y(\Delta\eta_k+\mu_k\eta_k)+\sum_{j=0}^{k-1}\alpha_{k,j}(\mu_k(\eps)-\mu_j(\eps))\frac{\pa\varphi_{j,\eps}}{\pa y}
\end{equation} 
with the boundary conditions
\begin{equation}\label{sk-2-1}
\left\{
\begin{array}{ll}
r_{k,\eps}=0 & \text{ on } B_{I}\\
r_{k,\eps}=\eps\n f(x)\n w_{k,\eps}(x,\eps f(x)) {+ \mathcal{O}(\eps^3)}& \text{ on } B_{II}\\
\frac{\pa r_{k, \eps}}{\pa n} =0 & \text{ on } B_{III}
\end{array}
\right..
\end{equation}

Inductively, we assume that the Theorem~\ref{thm2} is true for $j\leq k-1$. Then by Lemma {~\ref{lemq3}} and Lemma~\ref{lemm1}, we have
\begin{equation} \label{sk-3} ||F_1||_{C^\alpha_\eps}=O({B_k}) 
\textrm{ and } ||F_2||_{C^\alpha_\eps}=O(\eps){,}\end{equation} 
{where}
\[
B_k=\eps^2+\sqrt{\eps A_k}.
\]

\begin{lemma} With the above notations, we have
$$ ||w_{k,\eps}||_{L^2({M_\eps})} =O(\eps^{1/2}B_k) \textrm{ and }  ||\tilde\nabla w_{k,\eps}||_{L^2({M_\eps})} =O(\eps^{1/2}B_k).$$
\end{lemma}

\begin{proof}   Multiplying both sides of ~\eqref{sk-1} by $w_{k,\eps}$ and integrating by parts, using~\eqref{sk-3}, we get
\begin{align}\label{poiu}
\begin{split}
&
||\n w_{k,\eps}||^2_{L^2({M_\eps})} -\mu_{k}(\eps) ||w_{k,\eps}||^2_{L^2({M_\eps})} \\
&\leq CB_k\cdot||w_{k,\eps}||_{L^1(M_\eps)}+\left|\int_{B_{II}}w_{k,\eps}\frac{\pa w_{k,\eps}}{\pa n}\right|.
\end{split}
\end{align}
{By} ~\eqref{5-7}, we have
$$
\left|\int_{B_{II}}w_{k,\eps}\frac{\pa w_{k,\eps}}{\pa n}\right|=\left|\int_{B_{II}}w_{k,\eps}\frac{\pa U_k}{\pa n}\right|  \leq C \eps^{5/2}( ||w_{k,\eps}||_{L^2({M_\eps})}+ ||\n w_{k,\eps}||_{L^2({M_\eps})}).$$
Thus we have
\[
||\n w||^2_{L^2({M_\eps})} -\mu_{k}(\eps) ||w||^2_{L^2({M_\eps})}\leq C\eps^{5/2}( ||w_{k,\eps}||_{L^2({M_\eps})}+ ||\n w_{k,\eps}||_{L^2({M_\eps})}).
\]

By the Poincar\'e inequality, we have
\[
\mu_{k+1}(\eps)\int_{M_\eps}|w_{k,\eps}|^2\leq
\int_{M_\eps}|\n w_{k,\eps}|^2.
\]
The lemma is proved since by our ``generic'' assumption, there is a gap between $\mu_{k+1}(\eps)$ and $\mu_k(\eps)$ that is independent of $\eps$.
\end{proof} 

Using the above lemma we shall complete the proof of Theorem~\ref{thm1}.\\

{\bf Proof of Theorem~\ref{thm1}.} Since $B_k=O(\eps)$, by the above lemma, $||w_{k,\eps}||_{L^2(M_\eps)}=O(\eps^{3/2})$. It follows that
\[
A_k=\sum_{j=1}^k\int_{M_\eps}\left|\frac{\pa \vphi_{k,\eps}}{\pa y}\right|^2=O(\eps^3),
\]
and thus Theorem~\ref{thm1} follows from Lemma~\ref{lemq3}.
\qed

\begin{cor}\label{7-1}
With  the above notations, we have
$$ ||w_{k,\eps}||_{L^2({M_\eps})} =O(\eps^{5/2}) \textrm{ and } ||\tilde \nabla w_{k,\eps}||_{L^2({M_\eps})} =O(\eps^{5/2}).$$
\end{cor}

\qed

Now we work towards the proof of Theorem~\ref{thm2}.

\begin{lemma}$w_{k,\eps}=O(\eps^2)$
\end{lemma}
\begin{proof}  {We prove the lemma using Moser iteration.}  By ~\eqref{sk-3} {and Theorem ~\ref{thm1},} $F_1=O(\eps^{2})$ is a bounded function.  Using integration by parts, we get
\[
p\int_{M_\eps}|w_{k,\eps}|^{p-1}|\nabla w{_{k, \eps}}|^2\leq\int_{M_\eps}|F_1||w_{k,\eps}|^p+\mu_k(\eps)\int_{M_\eps}|w_{k,\eps}|^{p+1}
+\int_{B_{II}}|w_{k,\eps}|^p\left|\frac{\pa w_{k,\eps}}{\pa n}\right|.
\]
By~\eqref{eq:4}, we have
\[
\int_{B_{II}}|w_{k,\eps}|^p\left|\frac{\pa w_{k,\eps}}{\pa n}\right|\leq C\eps^3\int_{B_{II}}|w_{k,\eps}|^p.
\]
By~\eqref{5-7}, we have
\[
\int_{B_{II}}|w_{k,\eps}|^p\leq 
\left(\frac 1\eps\int_{M_\eps}|w_{k,\eps}|^p+p\sqrt{\int_{M_\eps}|w_{k,\eps}|^{p-1}}\cdot
\sqrt{\int_{M_\eps}|w_{k,\eps}|^{p-1}|\n w_{k,\eps}|^2}\right).
\]
It follows that 
\[
\int_{B_{II}}|w_{k,\eps}|^p\left|\frac{\pa w_{k,\eps}}{\pa n}\right|\leq \frac 14 p\int_{M_\eps}|w_{k,\eps}|^{p-1}|\nabla w_{k,\eps}|^2
+\eps^6 p\int_{M_\eps} |w_{k,\eps}|^{p-1}+\eps^2\int_{M_\eps}|w_{k,\eps}|^p.
\]
By the Young's inequality, we have
\[
\eps^2|w_{k,\eps}|^p+\eps^6|w_{k,\eps}|^{p-1}\leq |w_{k,\eps}|^{p+1}+\eps^{2(p+1)}.
\]
Thus we have
\[
\int_{B_{II}}\left|w_{k,\eps}\right|^p\left|\frac{\pa w_{k,\eps}}{\pa n}\right|\leq
\frac p4 \int_{M_\eps}|w_{k,\eps}|^{p-1}|\nabla w_{k,\eps}|^2+Cp\left( {\left(\int_{M_\eps}|w_{k,\eps}|^{p+1}\right)}+\eps^{2(p+1)+1}\right),
\]
from which  we have
\[
p\int_{M_\eps}|w_{k,\eps}|^{p-1}\left|\nabla w_{k,\eps}\right|^2\leq Cp\left ({\left( \int_{M_\eps}|w_{k,\eps}|^{p+1}\right)}+\eps^{2(p+1)+1}\right).
\]
Using the above inequality and the Sobolev inequality~\eqref{kl-1}, we have
\[
\left(\int_{M_\eps}|w_{k,\eps}|^{(p+1)\frac{n+1}{n-1}}\right)^{\frac{n-1}{n+1}}\leq Cp^2\eps^{-\frac{2}{n+1}} \left(\int_{M_\eps}|w_{k,\eps}|^{p+1}+\eps^{2(p+1)+1}\right).
\]
We thus have
\[
\left(\int_{M_\eps}|w_{k,\eps}|^{(p+1)\frac{n+1}{n-1}}+\eps^{2(p+1)\frac{n+1}{n-1}+1}\right)^{\frac{n-1}{n+1}}\leq Cp^2\eps^{-\frac{2}{n+1}} \left(\int_{M_\eps}|w_{k,\eps}|^{p+1}+\eps^{2(p+1)+1}\right).
\]

Let 
\[
a_k=2\left(\frac{n+1}{n-1}\right)^k
\]
for $k\geq 0$.
Let
\[
b_k=\left(\int_{M_\eps}|w_{k,\eps}|^{a_k}+\eps^{2(a_k)+1}\right)^{1/a_k}.
\]
Then we have
\[
b_{k+1}\leq (C(a_k)^2\eps^{-\frac{2}{n+1}})^{1/a_k} b_k.
\]
The standard iteration process shows that
\[
|| w_{k,\eps}||_{C^0}\leq C \eps^{-\frac{1}{n+1}\sum 1/a_k}b_0\leq C\eps^{2}.
\]
\end{proof}

\begin{lemma} \label{lemq9}
$r_{k,\eps}=O(\eps^2)$.
\end{lemma}

\begin{proof} 
We could run Moser iteration again to get the estimate. However, the following proof using the maximum principle seems to be simpler.  Using~\eqref{s-1}, ~\eqref{sk-3}, and the boundary conditions~\eqref{sk-2-1}, we have
\[
||w_{k,\eps}||_{C^{2,\alpha}_\eps}\leq\eps^2,
\]
which implies that
\[
|\tilde\n^2 w_{k,\eps}|\leq C,\qquad |\tilde\n w_{k,\eps}|\leq C\eps
\]
by the interpolation inequalities.
For $C>0$ large, by~\eqref{sk-3}, we have
\[
\tilde\Delta (r_{k,\eps}+Cy^2)>0.
\]
Since $r_{k,\eps}=0$ on $B_I$, $\frac{\pa r_{k,\eps}}{\pa n}=0$ on $B_{III}$. By the maximum principle, we only need to estimate the maximum value of $r_{k,\eps}$ on $B_{II}$. 
But on $B_{II}$,  $r_{k,\eps}=\eps\n f\n w_{k,\eps}+O(\eps^3)=O(\eps^2)$. Thus we have $r_{k,\eps}\leq C\eps^2$.  Similarly estimating $r_{k,\eps}-Cy^2$ gives the other  side of the inequality.
\end{proof}

\begin{lemma} \label{lemq10} $||r_{k,\eps}||_{C_\eps^{2,\alpha}}=O(\eps^2)$, and 
$|\tilde\n ^2 w_{k,\eps}|=O(\eps)$.
\end{lemma}

\begin{proof} 
For fixed $y$, let $w=w_{k,\eps}$ and let $h=w(x,y f(x))$.  Then $h$ satisfies the equation
\begin{align}\label{s-4}
\begin{split}
&\Delta h+\mu_k(\eps)h\\
&=-\left(1-y^2|\n f|^2\right)\frac{\pa r_{k,\eps}}{\pa y}+2y\langle\n f, \n r_{k,\eps}\rangle+y r_{k,\eps}\Delta f+F_1(x,y f(x)).
\end{split}
\end{align}
Let $\Omega\subset M$ such that on $M\backslash\Omega$, {$f$ is identically equal to a positive constant $\delta$}. 
By the Schauder interior estimate, we have
\begin{align*}
&
||h||_{C^{3,\alpha}(\Omega)}\leq C(\eps^2+||r_{k,\eps}||_{C^{2,\alpha}}).
\end{align*}

Note that the above $C^{2,\alpha}$-norm is a function of $y$, and this norm is unscaled.  By Lemma~\ref{lemq9}, ~\eqref{sk-2} and (\ref{sk-3}) the global Schauder estimate ~\eqref{s-2} gives
\begin{align*}
&
||r_{k,\eps}||_{C_\eps^{2,\alpha}}\leq C(\eps^2+||\eps\n f\n h||_{C^{2,\alpha}_\eps} )\leq C(\eps^2+||h||_{C_\eps^{3,\alpha}(\Omega)}).
\end{align*}
The relation between weighted and the usual H\"older norms is (up to a constant)
\[
\eps^{k+\alpha}||u||_{C^{k,\alpha}}\leq ||u||_{C_\eps^{k,\alpha}}\leq \eps^{k+\alpha}||u||_{C^{k,\alpha}}+||u||_{C^0}.
\]
Thus we have
$$||  r_{k,\eps}||_{C_\eps^{2,\alpha}}\leq C(\eps^2+\eps^{3+\alpha}||h||_{C^{3,\alpha}(\Omega)}+||h||_{C^0(\Omega)})$$
$$\leq C(\eps^2+\eps^{3+\alpha}(\eps^2+||r_{k,\eps}||_{C^{2,\alpha}})) \leq C(\eps^2+\eps ||  r_{k,\eps}||_{C_\eps^{2,\alpha}}).$$
Therefore $||  r_{k,\eps}||_{C_\eps^{2,\alpha}}=O(\eps^2)$, which also  implies  that 
\[
\frac{\pa^2 w_{k,\eps}}{\pa y^2}, \frac{\pa^2 w_{k,\eps}}{\pa x_j \pa y}=O(\eps),
\]
where $(x_1,\cdots,x_n)$ is any {local} coordinate system on $M$.
Using the global Schauder estimate on~\eqref{s-4} again, we get
\[
\frac{\pa^2 w_{k,\eps}}{\pa x_j \pa x_i}=O(\eps).
\]
\end{proof}

\begin{lemma} 
$|\n r_{k,\eps}|=O(\eps^2)$.
\end{lemma}

\begin{proof} We need to prove that for any first order differential operator $R$ on $M$,
\[
R(r_{k,\eps})=O(\eps^2)
\]
uniformly for any $0\leq y\leq \eps$.
By Lemma~\ref{lemq10}, this is equivalent to 
\[
v=R(r_{k,\eps}(x,yf(x))) =O(\eps^2)
\]
uniformly for any $0\leq y\leq \eps$.

We first assume that the vector field $R$ is vertical to $\pa M$. Then by Lemma~\ref{lemq10},
$v=0$ on $B_I$, and 
\begin{equation}\label{tk-1}
v=R(\eps\n f(x)\n w_{k,\eps}(x,\eps f(x)))=O(\eps^2) \textrm{ on } B_{II}. 
\end{equation}
On $B_{III}$, since $R$ is vertical to $\pa M$, $v=g\frac{\pa}{\pa n}$ on $\pa M$ for some function $g$. Thus by Lemma~\ref{lemq10} again, 
\begin{equation}\label{tk-2}
v=g\frac{\pa}{\pa n} (r_{k,\eps}(x, y f(x)))=O(\eps^2).
\end{equation}
By~\eqref{sk-2} and Lemma~\ref{lemq10}, we have
\begin{equation}\label{tk-7}
\tilde\Delta R(r_{k,\eps})+\mu_k(\eps)  R(r_{k,\eps})=R(F_2)+[\Delta,R] r_{k,\eps}=O(1).
\end{equation}
Thus for $C>0$ large enough, we have
\[
\tilde\Delta (R(r_{k,\eps})+Cy^2)>0,
\]
and by the maximum principle, $R(r_{k,\eps})\leq O(\eps^2)$. Like in the proof of Lemma~\ref{lemq9}, the other side of the inequality can be obtained by estimating $R(r_{k,\eps})-Cy^2$.

Now we assume that $R$ is tangential on $\pa M$. We have similar estimates on $B_{I}$ and $B_{II}$ as above. On $B_{III}$, we note that
\begin{equation}\label{tk-3}
\frac{\pa}{\pa n} (R(r_{k,\eps}(x,y f(x))) =\tilde R(r_{k,\eps}(x,y f(x)))+O(\eps^2),
\end{equation}
where $\tilde R=[\frac{\pa}{\pa n}, R]$.
Let $\xi$ be a function on $M$ such that on $\pa M$, $\xi=0$; $\pa\xi/\pa n=1$; and on $M$,  $\eps\Delta\xi$ bounded; $\xi=O(\eps)$. We construct such a function $\xi$ as follows: let $\rho_(x)$ be a cut-off function on $\mathbb R$ such that $\rho\geq 0; \rho'(0)=0$ and $\rho(x)=0$ for $x\geq\eps$; $\eps|\rho_1|$ and $\eps^2|\rho_1''|$ are bounded.  Let 
$
\xi=-d\rho(d)$,
where $d$ is the distance function to the boundary $\pa M$.
Let
\[
\tilde v=R(r_{k,\eps}(x,y f(x)))-C_1\xi \max|\n r_{k,\eps}|+ C_2y^2
\]
for large $C_1, C_2$. By choosing $C_1$ large enough, from~\eqref{tk-3}, we have
\[
\frac{\pa\tilde v}{\pa n}<0
\]
on $B_{III}$. Fixing $C_1$, we choose $C_2$ large enough. Then by Lemma~\ref{lemq9}, we have
$
\tilde\Delta\tilde v>0$.
By the maximum principle, the maximum point of $\tilde v$ must be reached on $B_I\cup B_{II}$.  By the boundary conditions~\eqref{sk-2-1}, we have
\[
\tilde v\leq C(\eps^2+\eps\max|\n r_{k,\eps}|).
\]
Thus we have
\[
R(r_{k,\eps}(x,yf(x)))\leq C(\eps^2+\eps \max|\n r_{k,\eps}|).
\]
Since $R$ is arbitrary, this yields
\[
\max|\n r_{k,\eps}|\leq C(\eps^2+C_1\delta \max|\n r_{k,\eps}|).
\]
\end{proof}

{\bf Proof of Theorem~\ref{thm2}.} 
The proof is similar to that of Lemma~\ref{lemq10}. On $B_I$, $R(r_{k,\eps})=0$; on $B_{II}$, by~\eqref{tk-1}, we have
\begin{equation}\label{uk-1}
||R(r_{k,\eps})||_{C_\eps^{2,\alpha}(B_{II})}\leq C\eps^2+C\eps^{3+\alpha}||h||_{C^{4,\alpha}(\Omega)}.
\end{equation}
On $B_{III}$, we have two cases.  If $R$ is vertical to $\pa M$, by~\eqref{tk-2},
\[
||R(r_{k,\eps})||_{C_\eps^{2,\alpha}(B_{III})}\leq C(\eps^2+\eps||\n r_{k,\eps}||_{C_\eps^{2,\alpha}}).
\]
If $R$ is tangential to $\pa M$, by~\eqref{tk-3},
\[
\left|\left|\frac{\pa}{\pa n} R(r_{k,\eps})\right|\right|_{C_\eps^{1,\alpha}(B_{III})}\leq C\eps^2+||\n r_{k,\eps}||_{C_\eps^{1,\alpha}}.
\]
By the Schauder estimate and ~\eqref{tk-7}, 
\begin{equation}\label{uk-3}
||R(r_{k,\eps})||_{C_\eps^{2,\alpha}}\leq C(\eps^2+\eps^{3+\alpha}||h||_{C^{4,\alpha}(\Omega)}+||\n r_{k,\eps}||_{C_\eps^{1,\alpha}}).
\end{equation}
By the Schauder interior estimate, 
\begin{equation}\label{uk-2}
||h||_{C^{4,\alpha}(\Omega)}\leq C(\eps^2+||r_{k,\eps}||_{C^{3,\alpha}})\leq C(\eps^2+\eps^{-2-\alpha}||\n r_{k,\eps}||_{C_\eps^{2,\alpha}}).
\end{equation}
Combining~\eqref{uk-1},~\eqref{uk-3},~\eqref{uk-2}, we get
\[
||\n r_{k,\eps}||_{C_\eps^{2,\alpha}}\leq C(\eps^2+\eps||\n r_{k,\eps}||_{C_\eps^{2,\alpha}}+||\n r_{k,\eps}||_{C_\eps^{1,\alpha}}),
\]
which implies the theorem by the interpolation inequalities.
\qed

{\section*{Acknowledgments} 
The first author is supported by NSF grant DMS-09-04653.  The second author gratefully acknowledges the support of the Hausdorff Center for Mathematics and {the Max Planck Institut f\"ur Mathematik in Bonn.  Both authors would like to thank the anonymous referee for comments which improved the quality of the paper.}} 

\bibliographystyle{}
\bibliography{}

\end{document}